\documentclass[10pt,a4paper]{article}

\usepackage{cite}                                       
\usepackage[latin1]{inputenc}
\usepackage[english]{babel}
\usepackage{amssymb,amsmath}
\usepackage{a4wide}
\usepackage[colorinlistoftodos]{todonotes}
\usepackage{enumitem}

\usepackage{color}

\newcommand{\M}{\mathcal M}

\newcommand{\Z}{\mathcal Z}
\newcommand{\Ph}{\mathcal P}

\newcommand{\T}{\mathcal T}
\newcommand{\Sph}{{\mathcal S}}
\newcommand{\W}{\mathcal W}
\newcommand{\C}{\mathcal C}
\newcommand{\U}{\mathcal U}
\newcommand{\Orb}{{\mathsf O}}
\newcommand{\Cst}{\mathsf{C}}

\newcommand{\Xscr}{\mathcal{X}}
\newcommand{\Dscr}{\mathcal{D}}

\def\RR{\mathbb{R}}
\def\ZZ{\mathbb{Z}}

\def\Id{\mathbb I}

\def\eps{\epsilon}
\def\Im{i}
\def\lie#1{L_{#1}}
\def\chiph{\chi^{\phantom s}}
\def\gt{>}
\def\lt{<}
\def\le{\leq}
\def\ge{\geq}

\def\Chi{\Xscr}

\newcommand{\tond}[1]{{\left(#1\right)}}
\newcommand{\quadr}[1]{{\left[#1\right]}}
\newcommand{\inter}[1]{{\langle#1\rangle}}
\newcommand{\graff}[1]{\{#1\}}
\newcommand{\qed}{{\vskip-18pt\null\hfill$\square$\vskip6pt}}
\newcommand{\Ham}[2]{H^{(#1)}_{#2}}
\newcommand{\Poi}[2]{\{#1,#2\}}
\newcommand{\derp}[2]{{\frac{\partial #1}{\partial #2}}}

\newcommand{\ncamp}[1]{\Big\|#1\Big\|^\oplus}
\newcommand{\norm}[1]{\left\|#1\right\|}

\newtheorem{theorem}{Theorem}[section]
\newtheorem{theorem*}{Theorem}
\newtheorem{lemma}{Lemma}[section]
\newtheorem{corollary}{Corollary}[section]
\newtheorem{proposition}{Proposition}[section]
\newtheorem{definition}{Definition}[section]

\newenvironment{proof}[0]{\noindent{\bf proof:}}{$\square$\par\medskip}

\title{Long time stability of small amplitude Breathers\\
  in a mixed FPU-KG model}

\author{
  Simone Paleari\thanks{Universit\`a degli Studi di Milano,
    Dipartimento di Matematica, Via Saldini 50, 20133 Milano (Italy)}
  \footnote{{\tt simone.paleari@unimi.it}} 
  \hskip6pt and Tiziano Penati$\null^*$
  \footnote{{\tt tiziano.penati@unimi.it}}
}

\begin{document}

\maketitle

\begin{abstract}
  In the limit of small couplings in the nearest neighbor interaction,
  and small total energy, we apply the resonant normal form result of
  a previous paper of ours to a finite but arbitrarily large mixed
  Fermi-Pasta-Ulam Klein-Gordon chain, i.e. with both linear and
  nonlinear terms in both the on-site and interaction potential, with
  periodic boundary conditions.
  
  An existence and orbital stability result for Breathers of such a
  normal form, which turns out to be a generalized discrete Nonlinear
  Schr\"odinger model with exponentially decaying all neighbor
  interactions, is first proved.

  Exploiting such a result as an intermediate step, a long time
  stability theorem for the true Breathers of the KG and FPU-KG
  models, in the anti-continuous limit, is proven.
\end{abstract}


\section{Introduction and statement of the results}
\label{s:1}

We consider a mixed Fermi-Pasta-Ulam Klein-Gordon model (FPU-KG) as
described by the following Hamiltonian
\begin{align}
  \label{e.H}
  H(x,y) &= \frac12\sum_{j=1}^N \left[ y^2_j + x^2_j + a(x_{j+1}-x_j)^2
    \right] + \frac14\sum_{j=1}^N\quadr{x_j^4 +
    {b}(x_{j+1}-x_j)^4}\ ,\\  x_0 &= x_N \ , \qquad y_0 = y_N \ ,
\end{align}
i.e. a finite chain of $N$ degrees of freedom and periodic boundary
conditions, where $a>0$ and $b\geq 0$ are the linear and nonlinear
coupling coefficients. It can be remarked that the classical KG model
($a\neq0$, $b=0$) is included as a particular case, but the pure FPU
one is clearly not covered\footnote{We can't include the FPU case
  (which corresponds to $a\neq0$ and $b\neq0$, without the on-site
  potentials $x_j^2$ and $x_j^4$) since the normal form construction
  we rely on does not apply for such a model; moreover, in the present
  context of an application to the stability of Breather solutions,
  the discussion of such an extension could be pointless since in many
  FPU-like models, like in~\cite{JamKC13} no fundamental Breathers
  exist.}.

According to a previous result of ours, for any $r\geq 1$, provided
the coupling parameters $a$ and $b$ are correspondingly small enough,
there exists a canonical transformation $T_\Chi$ which puts the
Hamiltonian \eqref{e.H} into an extensive resonant normal form of
order $r$
\begin{equation*}
    \Ham{r}{} = H_\Omega + \Z + {P^{(r+1)}}\ ,
    \qquad\qquad
    \{H_\Omega,\Z\}=0 \ .
\end{equation*}
with $H_\Omega$ a system of $N$ identical oscillators whose frequency
$\Omega$ turns out to be the average of the linear spectrum of the
original Hamiltonian, $\Z$ a non-homogeneous polynomials of order
$2r+2$, $P^{(r+1)}$ a remainder of order at least $2r+4$ (see Theorem
\ref{prop.gen} in Section~\ref{s:2}). Strictly speaking, the original
statement is formulated in the case $b=0$, see~\cite{PalP14}, but
holds also for $b\neq0$ and small.

The above normal form was indeed shown in~\cite{PalP14} to be well
defined in a small ball $B_R(0)$ both in euclidean and in supremum
norm, i.e. both in a regime of finite total and respectively
specific\footnote{{We recall that, given $E=H(x,y)$ the total energy,
    the specific energy is the average energy per degree of freedom
    $E/N$.}} energy; one of the key points was indeed to consider
finite but arbitrarily large systems (along a direction we followed
also, e.g., in~\cite{PP12}), with estimates uniform in the size of the
chain, hence valid in the limit $N\to+\infty$. However, only in the
case of the euclidean norm the almost invariance of $H_\Omega$ over
times $|t|\sim R^{-r-1}$ was granted, due to the equivalence between
$H_\Omega$ and the selected norm. Moreover, looking at the structure
of $\Z$, the normal form $H_\Omega+\Z$ appears as a generalized
discrete nonlinear Schr\"odinger (GdNLS\footnote{Recently several
  works appeared on GdNLS models with more than first neighbor
  interactions, like~\cite{KouKCR13,ChoCMK11}, or with higher
  nonlinearities, like~\cite{ChoP11,CarTCM06}, where spatially
  localized periodic orbits, like breathers or multibreathers, are
  studied. {Please remark that we here use the term ``generalized''
    exactly to indicate {\sl a generalization}, without any particular
    reference to other ``generalized'' dNLS.}})  chain: it is
characterized by all neighbors couplings, with exponential decay of
the coefficients with the distance between sites, both in the linear
and nonlinear terms, the last ones being of order $2r+1\geq 3$.  Since
the Hamiltonian of such a normal form is given by an expansion both in
energy, through the degree of the polynomials, and in coupling, it is
actually cumbersome and somewhat useless to give here a complete and
explicit formulation; the following are the leading terms, in the
transformed variables $(\tilde x,\tilde y)$
  \begin{align*}
    H_{\rm GdNLS} = 
    & \sum_{j=1}^N \bigg[\frac{\Omega}2 ( {\tilde y}_j^2 + {\tilde x}_j^2 ) \; + 
      && H_\Omega
\\
    +\; &  \mathcal{O}(c)
      ({\tilde x}_j{\tilde x}_{j+1} + {\tilde y}_j{\tilde y}_{j+1}) +
    \mathcal{O}(c^2)
      ({\tilde x}_j{\tilde x}_{j+2} + {\tilde y}_j{\tilde y}_{j+2}) +
    \mathcal{O}(c^3) \; +
    &&\Z\text{: quadratic part} 
\\ 
    +\; & \mathcal{O}(c^0) ({\tilde x}_j^2+{\tilde y}_j^2)^2
    + \mathcal{O}(c)
      ({\tilde x}_j^2+{\tilde y}_j^2)({\tilde x}_j{\tilde x}_{j\pm 1}+
       {\tilde y}_j{\tilde y}_{j\pm 1})
    + \mathcal{O}(c^2) \Big]  +
    &&\Z\text{: quartic part}
\\ 
    +\; &\ldots   &&\Z\text{: higher orders}
\end{align*}
where we have introduced the collective coupling constant
\begin{equation}
  \label{e.c}
  c:=\max\{a,b\} \ .
\end{equation}
If one truncates the above Hamiltonian, using only $H_\Omega$ and the
first term of both the quadratic and quartic part of $\Z$, it is
possible to recognize the usual dNLS, here written in real
coordinates.

In this work we exploit the invariance of $H_\Omega$ in the above
resonant normal form part, in order to get some stability results
about true or approximated Breather solution in the model
\eqref{e.H}. The results we are going to present hold in the small
total energy $E<E_*(r)$ regime and for $c<c_*(E,r)$ small enough,
hence in the anti-continuous limit.

For a more precise formulation, let us give some notation. We denote
with $\Ph$ the phase space $\RR^{2N}$ endowed by the usual euclidean
norm $\norm{z}$, where $z=(x,y)$ is the generic element of $\Ph$.
Given $z\in\Ph$, let us also denote by ${\Orb}(z)$ the orbit through
$z$. We also need to introduce a suitable ``orbital'' distance. We use
the Hausdorff distance $d_H$, which is a metric once restricted to the
subset of non empty and compact sets: since we consider periodic
orbits and subsets of orbits parametrized by a closed interval of
time, $d_H$ satisfies all the relevant properties we need. We recall
the definition: given two subsets $A$ and $B$ of $\Ph$,
\begin{equation}
  \label{e.dist.def}
  \begin{aligned}
    d(A,B) &:= \sup_{a\in A}\inf_{b\in B}\norm{a - b} \ ,
  \\
    d_H(A,B) &:= \max\{d(A,B),d(B,A)\} \ .
  \end{aligned}
\end{equation}

Let us denote\footnote{We will follow the following convention to
  denote profiles, and consequently orbits via the symbol
  ${\Orb}(\cdot)$:
\begin{align*}
\phi &&& \text{generic profile for FPU-KG}
     &&&&&&&& z=(x,y) && \text{original coordinates}
\\
\Psi &&& \text{Breather profile for FPU-KG}
     &&&&&&&& \tilde z=(\tilde x,\tilde y) && \text{transformed coordinates}
\\
\psi &&& \text{Breather profile for GdNLS}
     &&&&&&&& \tilde \psi && \text{objects in transformed coordinates}
\end{align*}
See also \eqref{e.inv.psi.ab} for the relation between an object in
original coordinates and in transformed ones.} by $\Psi_{a,b}$ and
${\Orb}(\Psi_{a,b})$ respectively the Breather initial profile and its
orbit for our FPU-KG model~\eqref{e.H}. The existence of such an
object in the anti-continuous limit (i.e. as a family in $(a,b)$
emerging from the trivial one-site excitation solution available when
$(a,b)=(0,0)$) has been obtained originally in \cite{MacA94} (strictly
speaking in the case $b=0$).

In the formulation below of our result, the long stability time
$T_{\eps,r,R}$ (see~\eqref{e.times}) scales as
\begin{equation}
\label{e.appr.tscale}
T_{\eps,r,R} \simeq \eps^2 {\frac{(Rr)^{-2r}}{R^4}} \ ,
\end{equation}
where $r$ is the aforementioned normal form order, $R$ control the
small energy of the objects involved, and $\eps$, sufficiently smaller
than $R$, is the (tunable part of the) radius of the stability
neighborhood.

Finally, to unambiguously fix the two-parameter family $\Psi_{a,b}$,
we require it to emerge, in the anti-continuous limit, from the
single-site oscillator with prescribed amplitude
$\norm{\Psi_{0}}=R/6$, where here and in the following, we will use a
single sub-scripted $0$ to indicate the values $(a,b) = (0,0)$
(see~\eqref{e.def.psi0} and~\eqref{e.def.Psi0} for explicit
definitions).


\begin{theorem}
\label{t.b.kg}
Fix an arbitrary integer $r\geq1$. Then there exists $R_*(r)<1$ such that
for all $R<R_*$ and $0<\eps \ll R^2$ there exist $c_*(r,R,\eps)$ and
$\delta(\eps)$, such that for all $c<c_*$ the (piece of) orbit
${\Orb}(\phi):=\{\phi(t) \ :\ |t| \leq T_{\eps,r,R} \,,\ \phi(0)=\phi
\}$, solution of~\eqref{e.H}, satisfies
\begin{equation}
\label{e.main}
  \norm{\phi-\Psi_{a,b}}<\delta
  \quad\Longrightarrow\quad
  d_H\tond{{\Orb}(\phi),{\Orb}(\Psi_{a,b})} < \eps
  \ .
\end{equation}
\end{theorem}

A first comment on the above statement pertains the coupling threshold
$c_*$ and its dependence on the relevant parameters. It depends on $r$
because of the normal form construction: the larger the transformation
steps number required, the smaller the perturbation parameter,
i.e. the coupling. The dependence from $R$ comes both from the normal
form procedure, when we need to control the size of the transformation
domains and the smallness of the remainder, and from the existence of
Breathers solutions of the GdNLS (this will be shown to be a necessary
intermediate step). The $\eps$ dependence appears instead in the last
part of the proof, when the distance between the Breather of the full
system and that of the normal form must be controlled.

Let us add some more details. As we said, our proof of
Theorem~\ref{t.b.kg} is based on the long time stability result of
Theorem~\ref{p.orb.contr}. We indeed first show the expected existence
and stability of a Breather for the normal form (GdNLS), respectively
by a continuation from the anti-continuous limit and exploiting the
second conserved quantity of the GdNLS. Let us denote by
${\Orb}(\psi_{a,b})$ the orbit of such a Breather, emerging from the
same single-site oscillator $\Psi_0$ introduced above. We remark that
the closed trajectory ${\Orb}(\psi_{a,b})$ represents an approximated
solution for~\eqref{e.H}.  We then use the small remainder given by
the normal form transformation to translate the infinite time
stability of the GdNLS dynamics around the GdNLS Breather
${\Orb}(\psi_{a,b})$ into a long time stability of the FPU-KG dynamics
around the same object. This concludes the sketch of the proof of
Theorem~\ref{p.orb.contr}, where a stability control of the FPU-KG
dynamics can be obtained in the form
\begin{equation*}
  \norm{\phi-\psi_{a,b}}<\delta
  \quad\Longrightarrow\quad
  d_H\tond{{\Orb}(\phi),{\Orb}(\psi_{a,b})} < \eps
\ ,
\end{equation*}
for $|t|\lesssim T_{\eps,r,R}$ and $c<c_*$, in this case with
$c_*(r,R)$ independent of $\eps$. Thus, at fixed $r$ and $R$, one can
play\footnote{However, in order to get a meaningful result, $\eps$
  shouldn't be taken too small: otherwise the stability time
  $T_{\eps,r,R}$, which scales as $\eps^2$, could fall shorter than
  the period of the (true/approximated) Breather.}  with $\eps$ to
strengthen the stability control without further requirements on the
couplings.

To get the result of Theorem~\ref{t.b.kg} one eventually exploits the
closeness of the FPU-KG Breather $\Orb(\Psi_{a,b})$ to the GdNLS
Breather $\Orb(\psi_{a,b})$: both the objects emerge in the
anti-continuous limit from the same configuration $\Psi_{0}$, thus
using the continuity in their (common) continuation parameter $c$ one
gets a (weak) closeness of order $\sqrt[4]{c}$. Here enters the
dependence of $c_*$ also on $\eps$: this is needed in order to ensure
that the true Breather configuration $\Psi_{a,b}$ lies well within the
stability basin of the approximated Breather orbit
$\Orb(\psi_{a,b})$. Furthermore, in order to transfer the stability of
$\Orb(\psi_{a,b})$ to the stability of $\Orb(\Psi_{a,b})$ the
triangular inequality of the Hausdorff metric $d_H$ is also needed.

We would like to stress that our result resembles, in its formulation
and strategy, ~Theorem~2.1 of~\cite{Bam96}, which is the first, and
actually one of the few, result of long time stability of Breathers
for weakly coupled oscillators (see also \cite{Bam98}): indeed,
although Nekhoroshev-type stability was expected since the earliest
papers (see, e.g. \cite{MacA94}), most of the literature on the
stability of Breathers (and of their multi-site generalizations,
called Multibreathers) deals with the linear stability (see
\cite{Aub97,PelKF05,PelS12,Yos12}).

There are however some differences with \cite{Bam96}, that we would
like to underline here. The first one is that in \cite{Bam96} the
closeness to the Breather solutions was obtained with a ``local''
normal form around a generic an-harmonic oscillator (the system being
infinite), using only the linear coupling $a$ as small parameter
(since it treats the model \eqref{e.H} with $b=0$). As a consequence,
it is valid also for arbitrary large amplitudes and not only in the
small energy regime, like Theorem~\ref{t.b.kg}. Our normal form is
instead ``more global'', in the sense that it holds in a whole
neighborhood of the origin. Hence, within its limit of validity given
by the smallness of the energy, it can be used to capture the main
features of any Cauchy Problem.

Moreover, and differently from our Theorem~\ref{t.b.kg}, in
\cite{Bam96} the small parameter $a$ is used also in order to fix the
domain of stability: indeed, in~\cite{Bam96}, the corresponding of our
radius $\delta$ of the stability basin vanishes as ${a}\to 0$. This is
a consequence of the way the ``local'' normal form Theorem
(ref. Theorem 4.1 in \cite{Bam96}) has been used, choosing $\sqrt{a}$
as the size of the domain of validity, and it seems in contrast with
the intuition that by approaching the uncoupled system ($a=0$ in that
case, $c=0$ in the present one), the Breather should be increasingly
stable, not only in terms of time scale but also in terms of
domain. With respect to this aspect, our result is more flexible: as
already pointed, at fixed time scale (i.e. fixing $r$ and $E$) we are
allowed to arbitrarily decrease the coupling $a$ without shrinking the
stability basin.

Concerning instead the dependence on the coupling of the stability
time scale, the result in~\cite{Bam96} appears to be as strong as one
could hope, i.e. one has an exponential dependence of the form $T_a
\simeq \exp(a^{-1/6})$. Our result, on the contrary, seems to fail
completely in the expected growth of the time scale as the couplings
vanish, since neither $a$ nor $b$ appear explicitly
in~\eqref{e.appr.tscale}. However our result is indeed somewhat
similar once the implicit dependence on the couplings is taken into
account: the formulation of Theorem~\ref{t.b.kg} provides a stability
time $T_{\eps,r,R}$ which scales as a power of $(Rr)^{-1}$, which is
large provided the ``amplitude'' $R$ is sufficiently small with
respect to $1/r$ (see also condition~\eqref{e.R.sm1}), with an
exponent $r$ which can be arbitrarily increased by sufficiently
decreasing the coupling $c$. In the parameters plane $(r,c)$, the
allowed region has a border\footnote{We should also remark that
  different choices for $\sigma_1$ and $\sigma_*$, respectively right
  before and right after the statement of Theorem~\ref{prop.gen},
  could lead to a boundary of the form $cr^\alpha=\text{const}$, with
  $\alpha>1$. This could further improve the time scale dependence on
  $c$, at the price of lowering the thresholds of validity for the
  relevant parameters (see \cite{PalP14}, Section~4.2 for further
  details); we did not pursue an optimization of this type.}  roughly
described by $cr^4=\text{const}$. Thus one can either formulate the
statement, as we do, assuming an arbitrary $r$, provided $c$ is
smaller than something scaling as $1/r^4$; or one could fix $c$
(sufficiently small for independent reasons) and let $r$ up to
$1/\sqrt[4]{c}$. In the latter case, provided $R$ vanishes at least as
$\sqrt[4]{c}$, the stability time scale resemble very closely the
exponential one of~\cite{Bam96}. The price to be paid is that, the
smaller is the amplitude $R$, the smaller has to be the stability
domain parameter $\eps$.

There is a last comment in the comparison of our results with the
reference paper~\cite{Bam96}.  The stability in \cite{Bam96}, as we
said, is obtained through a normal form around the an-harmonic
oscillator which is going to constitute the core of the Breather,
actually by removing the dominant part of the coupling of such an
oscillator $(J,\varphi)$ with the rest of the chain: this typically
requires a Diophantine non resonance condition for the true frequency
$\omega(J)$ of the Breather with respect to its linear frequency
$\omega_0$. However, the smaller is $J$, the closer is $\omega(J)$ to
$\omega_0$ and thus proportionally smaller must be the parameter $\nu$
in the Diophantine condition $|k_1\omega+k_2\omega_0|\geq
\nu/|k|^2$. And this affects the small coupling interval $(0,a_*)$ for
which the result in \cite{Bam96} applies: indeed, since the normal
form construction needs $\sqrt{a}/\nu<1$ in order to be performed, the
threshold $a_*$ has to decrease at least like $\nu^2$, which means
$a_*\lesssim R^4$ in terms of small amplitude $R$. Our result is
instead completely free of any Diophantine condition on the Breather
frequency, implicitly requiring only non-resonance and non-degeneracy
of the frequency in order to have the existence of the Breather.  And
we stress that, even though we also require the coupling $c$ (and then
$a$) to be small enough with respect to the amplitude, as a sufficient
condition for the variational continuation from the anti-continuous
limit, our smallness condition is weaker, being of the order $c_*\ll
R^2$.

We conclude this Introduction by remarking that, since our strategy is
strongly based on a normal form construction for the quadratic part of
the Hamiltonian \eqref{e.H} (see also \cite{GioPP12,GioPP13}), it can
be applied also to different local nonlinearities, like for example
the Morse or the cubic potential in the DNA models
\cite{PeyB89,DauPW92}. Indeed, even the FPU-KG model presented here is
an easy extension with respect to the classical KG one, and we
included it here both to give a more general result and because we
were motivated by recent papers like~\cite{KarSKC13,Yos12}, where a
nonlinear quartic interaction is taken into account. Moreover, the
perturbation approach we exploited here, even simplified in its
preliminary step involving the quadratic part, can be applied to those
model where the coupling is purely nonlinear ($a=0$), thus justifying
the long time stability of compact-like Breathers (see
\cite{TchR99,RosS05}).

The paper is organized as follows. In Section~\ref{s:2} we reformulate
the normal form result (and the main ideas related) discussed in
\cite{PalP14}. In Section~\ref{s:3} we present and comment the two
results concerning the long time orbital stability of the approximated
and true breather solution.  A short Appendix contains the proofs of
the existence and stability of the GdNLS breather.

%
%
%
%

\section{Background:\! an extensive resonant normal form Theorem}
\label{s:2}

The aim of this Section is to present the resonant normal form Theorem
obtained in \cite{PalP14}, with a slightly different formulation which
is necessary to deduce Theorem \ref{t.b.kg}. At variance with respect
to the original paper, we here decided to select $r$, the order of the
normal form, as the main parameter used to express the thresholds of
validity of the construction, instead of the small couplings. Such a
different choice fixes the order of the normal form, hence its non
linear terms, leaving the small couplings as ``free'' parameters for
the continuation procedure from the anti-continuous limit. For the
above reasons, and in order to introduce some definitions and remarks
necessary for the comprehension of the perturbation part, in this
Section we also repeat, and slightly extend, some ingredients of
\cite{PalP14}.

\subsection{Extensivity}
\label{ss:form}

The perturbation construction developed in \cite{GioPP13,PalP14} is
strongly based on the property of extensivity typical of a class of
Hamiltonian like \eqref{e.H}: physically speaking, in all these models
the extensivity results from both the translation invariance and the
short interaction potentials. In particular, the extensivity allows to
sharply manage the dependence on the size of the system $N$ in the
estimates involved in the perturbation approach. We here recall a
possible formalization of this property, by means of the cyclic
symmetry, which has been already introduced, widely analyzed and then
exploited in \cite{GioPP12,GioPP13,PalP14}.

We denote by $x_j,\,y_j$ the position and the momentum of a particle,
with $x_{j+N} = x_j$ and $y_{j+N}=y_j$ for any $j$.

\paragraph{Cyclic symmetry. }
\label{p:cyclic}

We formalize the translation invariance by using the idea of
\emph{cyclic symmetry}. The \emph{cyclic permutation} operator $\tau$,
acting separately on the variables $x$ and $y$, is defined as
\begin{equation}
\label{e.perm}
\tau(x_1,\ldots,x_N) = (x_2,\ldots,x_N,x_1)\ ,\quad
\tau(y_1,\ldots,y_N) = (y_2,\ldots,y_N,y_1)\ .
\end{equation}
We extend its action on the space of functions as
\begin{equation*}
\bigl(\tau f\bigr)(x,y) = f(\tau(x,y)) = f(\tau x,\tau y) \ .
\end{equation*}

\begin{definition}
\label{d.cs}
We say that a function $F$ is \emph{cyclically symmetric} if
$\tau F = F$.
\end{definition}

We introduce now an operator, indicated by an upper index $\oplus$,
acting on functions: given a function $f$, a new function $F=
f^{\oplus}$ is constructed as
\begin{equation}
  F= f^{\oplus} := \sum_{l=1}^{N} \tau^l f \ .
\label{e.cycl-fun}
\end{equation}
We shall say that $f^{\oplus}(x,y)$ is generated by the \emph{seed}
$f(x,y)$.
We shall often use the convention of denoting cyclically symmetric
functions with capital letters and their seeds with the corresponding
lower case letter.

\paragraph{Polynomial norms. }
\label{p:polinorms}

Let $f(x,y)=\sum_{|j|+|k|=s} f_{j,k} x^j y^k$ be a homogeneous
polynomial of degree $s$ in $x,\,y$.  Given a positive $R$, we define
its polynomial norm as
\begin{equation}
\label{e.polinorm}
\|f\|_R := R^s \sum_{j,k} |f_{j,k}|\ .
\end{equation}

\paragraph{Norm of an extensive function.}
\label{p:norm-ext}

Assume now that we are equipped with a norm for our functions
$\norm{\cdot}$, e.g. the above defined polynomial norm. We introduce a
corresponding norm $\norm{\cdot}^\oplus$ for an extensive function
$F=f^\oplus$ by defining
\begin{equation}
\label{e.norm-germ}
\bigl\|F\bigr\|^{\oplus} = \|f\|\ ,
\end{equation}
i.e. by actually measuring the norm of the seed. Although the norm so
defined depends explicitly on the choice of the seed, this is harmless
in the perturbation estimates since
\begin{displaymath}
\|F\| \le N \bigl\|F\bigr\|^{\oplus} = N \norm{f}\ ,
\end{displaymath}
for any $f$ such that $F=f^\oplus$.

\paragraph{Circulant matrices.}
\label{p:circulant}

Dealing with particular functions which are quadratic forms, the
cyclic symmetry coming from extensivity assumes a particular form.
Let us thus restrict our attention to the harmonic part of the
Hamiltonian: it is a quadratic form represented by a matrix $A$
\begin{equation}
\label{e.Aintro}
H_0(x,y) = \frac12 y\cdot y + \frac12 Ax\cdot x.
\end{equation}
If the Hamiltonian $H$ is extensive, then $H_0=h_0^\oplus$.  This
implies that $A$ commutes with the matrix $\tau$ representing the
cyclic permutation~\eqref{e.perm}
\begin{equation}
\label{e.tau}
\tau_{ij}=
\begin{cases}
1\quad {\rm if}\ i=j+1\>({\rm mod}\,N)\>,\\
0\quad \rm{otherwise}.
\end{cases}
\end{equation}
We remark that the matrix $\tau$ is orthogonal and generates a cyclic
group of order $N$ with respect to the matrix product.

%

In our problem the cyclic symmetry of the Hamiltonian implies that the
matrix $A$ of the quadratic form is circulant. Obviously it is also
symmetric, so that the space of matrices of interest to us has
dimension~$\quadr{\frac{N}2}+1$. Indeed, a circulant and symmetric
matrix is completely determined by $\quadr{\frac{N}2}+1$ elements of
its first line.

\def\supp{\mathop{\rm supp}}
\def\diam{\mathop{\rm diam}}
\def\corsivo#1{{\sl #1}}

\paragraph{Interaction range}
\label{ss.int.range}

We give here a formal characterization of finite and short range
interaction, pointing out some properties that will be useful in the
rest of the paper. We restrict our analysis to the set of polynomial
functions.  We start with some definitions. Let us label the variables
as $x_l,y_l$ with $l\in\ZZ$, and consider a monomial $x^jy^k$ (in
multiindex notation).
\begin{definition}
We define the \emph{support} $S(x^jy^k)$ of the monomial and the
\emph{interaction distance} $\ell(x^jy^k)$ as follows: considering the
exponents $(j,k)$ we set
\begin{equation}
\label{e.supp}
S(x^jy^k) = \{l\>:\>j_l\neq0 {\rm\ or\ } k_l\neq0 \}\ ,\quad
\ell(x^jy^k) = \diam\bigl(S(x^jy^k)\bigr) \ .
\end{equation}
We say that the monomial is \emph{left aligned} in case
$S(x^jy^k)\subset \{0,\ldots,\ell(x^jy^k)-1\}$.
\end{definition}
The definition above is extended to a homogeneous polynomial $f$ by
saying that $S(f)$ is the union of the supports of all the monomials
in $f$, and that $f$ is left aligned if all its monomials are left
aligned.  The relevant property is that if $\tilde f$ is a seed of a
cyclically symmetric function $F$, then there exists also a left
aligned seed $f$ of the same function $F$: just left align all the
monomials in $\tilde f$.

\paragraph{Short range (exponential decay of) interaction.}
\label{p:shortrange}
For the seed $f$ of a function consider the decomposition
\begin{equation}
\label{e.decomp}
f(z) =  \sum_{m\ge 0} f^{(m)}(z)\ ,\quad
 f^{(m)}(z) = \sum_{\ell(k)\le m} f_k z^k\ ,
\end{equation}
assuming that every $f^{(m)}$ is left aligned.
\begin{definition}
\label{d.class}
The seed $f$ (of an extensive function) is said to be of class
$\Dscr(C_f,\sigma)$ if
\begin{equation}
\label{dcdm.5}
\norm{f^{(m)}}_1 \le C_f e^{-\sigma m}\ ,\quad C_f\gt 0\,,\> \sigma\gt 0\ .
\end{equation}
\end{definition}

\paragraph{Continuity of extensive polynomials.} We add here some
regularity properties, which are absent in \cite{PalP14}.

\begin{lemma}
\label{l.cont.pol}
Any $F=f^\oplus$, polynomial of degree $m$, with
$f\in\Dscr(C_f,\sigma)$ is of class ${\cal C}^m(\RR^{2N},\RR)$, with
\begin{equation}
\label{e.cont.pol}
|F(z)| \leq \norm{F}^\oplus\norm{z}^m\ .
\end{equation}
\end{lemma}

\begin{proof}
Since
\begin{displaymath}
f(z) = \sum_{|k|=m} f_k z^k\ ,
\end{displaymath}
we have also
\begin{displaymath}
|F(z)| \leq \sum_{j=0}^{N-1}\sum_{|k|=m}|f_k| |z^k\circ \tau^j| \leq
\sum_{|k|=m}|f_k| \tond{\sum_{j=0}^{N-1}|z^k\circ \tau^j|} \leq
\norm{F}^\oplus\norm{z}^m\ .
\end{displaymath}
Any polynomial $F$ is represented by a symmetric ($m$) multilinear
operator $\hat F$, such that
\begin{displaymath}
\hat F(z,\ldots,z) = F(z)\ ,
\end{displaymath}
hence \eqref{e.cont.pol} gives
\begin{displaymath}
\sup_{\norm{z}\leq 1,\,z\not=0} |\hat F(z,\ldots,z)|\leq
\norm{F}^\oplus<\infty\ ,
\end{displaymath}
which is the continuity of $\hat F$. The continuity of the
differentials follows immediately from $\hat F$ being multilinear.
\end{proof}

\paragraph{Hamiltonian vector fields}

We consider, as an Hamiltonian, an extensive function $F$ with seed
$f$; we will make use of the common notation\footnote{For an easier
  notation we drop the Hamiltonian $F$ in the indexes of the
  components of the vector field.} $X_F=(X_1, \ldots,X_N,X_{N+1},
\ldots, X_{2N})$ to indicate the associated Hamiltonian vector field
$J\nabla F$, with $J$ given by the Poisson structure. The first easy,
but important, result is that also the Hamiltonian vector field
inherits, in a particular form, the cyclic symmetry; a possible choice
for the equivalent of the seed turn out to be the couple $(X_{1},
X_{N+1})$, i.e. the first and the $(N+1)^{\rm th}$ components of the
vector. This fact, which will be more clear thanks to the forthcoming
Lemma~\ref{l.seme.campo}, allows us to define in a reasonable and
consistent way the following norm
\begin{equation}
\label{e.def1}
\ncamp{X_F}_R := \norm{X_1}_R+\norm{X_{N+1}}_R \ .
\end{equation}

\begin{lemma}
\label{l.seme.campo}
Given $F=f^\oplus$, for the components of its Hamiltonian vector field
$X_F$ we have\footnote{An immediate consequence of
  \eqref{e.seme.campo} is that, defining the norm of the vector field
  as the sum of its components (i.e. a finite $\ell^1$ norm), we would
  get $\norm{X_F}_R = N\ncamp{X_F}_R$, which in turn justify the
  definition \eqref{e.def1}, and make it consistent with our previous
  definition~\eqref{e.norm-germ}.}
\begin{equation}
\label{e.seme.campo}
\begin{aligned}
 X_j &= \tau^{j-1} X_1
\cr
 X_{N+j} &= \tau^{j-1} X_{N+1}
\end{aligned}
\qquad\qquad
j=1,\ldots,N \ .
\end{equation}
Moreover, it holds
\begin{equation}
\label{e.xxx}
 \ncamp{X_F}_R = \sum_{l=1}^{2N}\norm{\derp{f}{z_l}}_R.
\end{equation}
\end{lemma}

%
%

%
%
%
%

\subsection{Resonant normal form}

In this part we recall, with a slightly different statement more based
on the parameter $r$, the resonant normal form result
of~\cite{PalP14}. Although the model \eqref{e.H} has an additional
nonlinear term, its main Theorem, here formulated as
Theorem~\ref{prop.gen} still apply, since it requires some decay
properties of the seeds of the Hamiltonian which are true also
for~\eqref{e.H}.

Indeed, we first recall the splitting of the Hamiltonian~\eqref{e.H}
as a sum of its quadratic and quartic parts $H=H_0+H_1$, i.e.
\begin{equation}
\label{e.H.dec}
 H_0(x,y) := \frac12\sum_{j=1}^N \quadr{y^2_j + x^2_j +
   a(x_j-x_{j-1})^2} \ , \qquad H_1(x,y) := \frac14\sum_{j=1}^N
 \quadr{x_j^4+ b(x_{j+1}-x_j)^4} \ .
\end{equation}

\subsubsection{Discussion on the small parameters}
\label{sss:small.par}

Since both $0<a<1$ and $0\leq b<1$ have to be considered as small
parameters, we define
\begin{equation}
\label{e.mu}
 \mu:=\sqrt[4]{\frac{2c}{1+2c}}\ ,
\end{equation}
where $c$ has been introduced in \eqref{e.c}: the new parameter $\mu$
will play the role of main perturbation parameter together with the
small radius $R$. Moreover, in order to deal with exponentially
decaying interactions (and to explain why we defined $\mu$ as we did),
let us introduce the following parameters, which will serve as decay
rates in the sense of Definition~\ref{d.class}
\begin{equation}
\label{e.sigma.ab0}
\sigma_a := \ln{\tond{\frac{1+2a}{2a}}}\ ,\qquad \sigma_b :=
\ln{\tond{\frac{1+2b}{2b}}}\ ,\qquad
\sigma_0:=\min\{\sigma_a,\sigma_b\}\ .
\end{equation}
As a consequence of \eqref{e.mu} and \eqref{e.sigma.ab0}, one has
\begin{equation}
\label{e.mu.sigma0}
\sigma_0 = \ln{\tond{\frac{1+2c}{2c}}}\ ,\qquad\qquad \mu =
e^{-\sigma_0/4}\ .
\end{equation}
It is important to notice that $H_1=h_1^\oplus$ with
$h_1\in\Dscr(C_{h_1},\sigma_b)$. Indeed by definition
\begin{displaymath}
H_1(x) = h_1^\oplus\ ,\qquad h_1:=\pm\frac{1}4x_0^4+
\frac{b}4(x_{1}-x_0)^4\ ;
\end{displaymath}
however by rearranging the monomials, it is possible to select $h_1$
as
\begin{displaymath}
h_1=h_1^{(0)} + h_1^{(1)}\ ,\qquad h_1^{(0)}=\frac{\pm1+2b}4x_0^4
\ ,\qquad h_1^{(1)}= \frac32b x_0^2x_1^2 - bx_0x_1(x_0^2+x_1^2)\ .
\end{displaymath}
with
\begin{displaymath}
\norm{h_1^{(l)}}\leq C_{h_1} e^{-\sigma_b l}\ ,\qquad l=0,1\ ,\qquad
C_{h_1}:=\frac74(1+2b)\ ,
\end{displaymath}

\subsubsection{Preliminary linear transformation}
\label{sss:lin.tr}

{We start performing the normalization process with} an
\emph{exponentially localized} linear transformation (see Proposition
2 of~\cite{PalP14}, and also~\cite{GioPP12,GioPP13}) to put the
quadratic part into a resonant normal form. Rewrite the matrix $A$,
introduced in~\eqref{e.Aintro}, as
\begin{equation}
\label{e.def-A}
A= (1+2a)\quadr{\Id - \frac{e^{-\sigma_a}}2(\tau + \tau^{\top})} \ ,
\end{equation}
which is clearly circulant and symmetric, and gives a finite range
interaction, in the form of a $e^{-\sigma_a}$ small perturbation of
the identity. Let the constant frequency $\Omega$ be the average of
the square roots of the eigenvalues of $A$, and take any $\sigma_1\in
(0,\sigma_0)$. We have (see mainly Proposition~3.1 in \cite{GioPP13}
and the related Proposition~1 in \cite{GioPP12})

\begin{proposition}
\label{p.1}
The canonical linear transformation $q=A^{1/4} x$, $p=A^{-1/4}y$ gives
the Hamiltonian $H_0$ the particular resonant normal form
\begin{equation}
\label{e.dec.H0}
 H_0 = H_\Omega + Z_0 \ ,
\qquad
 \Poi{H_\Omega}{Z_0}=0
\end{equation}
with $H_\Omega$ and $Z_0$ cyclically symmetric with seeds
\begin{equation*}
 h_\Omega = \frac\Omega2({\tilde x}_1^2+{\tilde y}_1^2) \ ,
\qquad
 \zeta_0\in\Dscr\bigl(C_{\zeta_0}(a),\sigma_0\bigr) \ ,
\end{equation*}
and transforms $H_1$ into a cyclically symmetric function with seed
\begin{equation*}
h_1\in\Dscr\bigl(C_{h_1}(a),\sigma_1\bigr) \ .
\end{equation*}
\end{proposition}

Some remarks are in order.
\begin{enumerate}

\item We first recall that it is exactly the above linear
  transformation which introduces the interaction among all sites,
  with an exponential decay with respect to their distance.

\item The original claim in \cite{GioPP13} would actually give
  $\zeta_0\in\Dscr(C_{\zeta_0}(a),\sigma_a)\subseteq
  \Dscr(C_{\zeta_0}(a),\sigma_0)$, since $\sigma_a\geq \sigma_0$. The
  choice of taking $\zeta_0\in\Dscr(C_{\zeta_0}(a),\sigma_0)$, thus
  loosing a bit of the exponential decay, is useful to simplify the
  control of the decay in the whole normal form construction.

\item As in the proof of Lemma~3.4 of \cite{GioPP13}, the decay rate
  $\sigma_1$ of the seed $h_1$ cannot be equal to that of the linear
  transformation, but can be chosen arbitrarily close, i.e. one has
  that $h_1\in\Dscr\bigl(C_{h_1}(a),\sigma_1\bigr)$ for any
  $\sigma_1<\sigma_0$. This is especially true when
  $\sigma_b>\sigma_a=\sigma_0$. We nevertheless make the following
  choice for $\sigma_1$
  \begin{equation}
    \label{e.sigma.1}
    \sigma_1 := \frac12\sigma_0\ ,
  \end{equation}
  once again, in order to simplify some calculations.

\end{enumerate}

\subsubsection{Normal form Theorem}
\label{sss:nft}

From now on we will simply indicate with $\Cst$ any constant which
does not depend on the relevant parameters, i.e. $R$, $r$ and
$c$. Consider the extensive Hamiltonian $H$ in the new ``exponentially
localized'' coordinates $(\tilde x,\tilde y)$, introduced by the
previous linear transformation
\begin{displaymath}
H = H_\Omega + Z_0 + H_1\ ;
\end{displaymath}
we have (see Theorem~1 in \cite{PalP14}):

\begin{theorem}
\label{prop.gen}
Consider the Hamiltonian $H=h^{\oplus}_{\Omega}+\zeta^{\oplus}_0 +
h^{\oplus}_1$ with seeds $h_{\Omega}=\frac{\Omega}{2}(x_0^2+y_0^2)$,
the quadratic term $\zeta_0$ of class $\Dscr(C_{\zeta_0},\sigma_0)$ with
$\zeta_0^{(0)}=0$, and the quartic term $h_1$ of class
$\Dscr(C_{h_1},\sigma_1\,)$. Pick a positive
$\sigma_0/4\leq\sigma_*<\sigma_1$, then there exist positive $\gamma$,
$r_{*}$ and $C_*$ such that for any positive integer $r$ satisfying
\begin{equation}
\label{e.muperr}
r\lt r_*\ ,
\end{equation}
there exists a finite generating sequence
$\Chi=\{\chi^{\oplus}_1,\ldots,\chi^{\oplus}_r\}$ of a Lie transform
such that $T_{\Chi}\Ham{r}{} = H$ where $\Ham{r}{}$ is an extensive
function of the form 
\begin{equation}
\label{e.Ham.r}
\Ham{r}{} = H_\Omega + \Z + {P^{(r+1)}}\ ,
\qquad\qquad
\begin{aligned}
\Z :&= Z_0 + \dots + Z_r
\\
\lie{\Omega}Z_s&=0 \ , \quad \forall s\in\{0,\ldots,r\} \ ,
\end{aligned}
\end{equation}
with $Z_s$ of degree $2s+2$ and $P^{(r+1)}$ a remainder starting with
terms of degree equal or bigger than $2r+4$.

Moreover, defining $C_r := 64r^2C_*$ and $\sigma_j := \sigma_1
-\frac{j-1}{r}(\sigma_1-\sigma_*)$, the following statements hold
true:
\begin{enumerate}[label=(\roman{*}), ref=(\roman{*})]
\item the seed $\chiph_s$ of $\Chi_s$ is of class
  $\Dscr(C_r^{s-1} \frac{C_{h_1}}{\gamma s}, \sigma_s)$;
\item the seed $\zeta_s$ of $Z_s$ is of class
  $\Dscr(C_r^{s-1}\frac{C_{h_1}}{s}, \sigma_s)$;
\item with the choice $\sigma_* = \sigma_0/4$, if it is satisfied the
  smallness condition on the total energy
  \begin{equation}
    \label{e.R.sm1}
    R<R^*:= \sqrt{\frac2{3(1+e)C_r}}\ ,
  \end{equation}
  then the generating sequence $\Chi$ defines an analytic canonical
  transformation on the domain $B_{\frac23 R}$ with the properties
  \begin{displaymath}
    B_{R/3}\subset T_\Chi B_{\frac23 R} \subset B_R\qquad\qquad B_{R/3}\subset
    T_\Chi^{-1} B_{\frac23 R} \subset B_R\ .
  \end{displaymath}
  Moreover, the deformation of the domain $B_{\frac23 R}$ is
  controlled by
  \begin{equation}
    \label{e.def.Tchi}
    z\in B_{\frac23 R}\qquad\Rightarrow\qquad \norm{T_\Chi(z)-z}\leq
    \Cst C_* R^3\ ,\qquad \norm{T^{-1}_\Chi(z)-z}\leq \Cst C_*
    R^3\ .
  \end{equation}
\item with the choice $\sigma_*=\sigma_0/4$, if it is satisfied
  \eqref{e.R.sm1}, the remainder is an analytic function on
  $B_{\frac23 R}$, and it is represented by a series of extensive
  homogeneous polynomials $\Ham{r}{s}$ of degree $2s+2$
  \begin{equation}
    \label{e.rem.r}
    P^{(r+1)} = \sum_{s\geq r+1}\Ham{r}{s}\qquad \Ham{r}{s}
    = \tond{h^{(r)}_s}^{\oplus}\ ,
  \end{equation}
  and the seeds $h^{(r)}_s$ are of class $\Dscr(2\tilde
  C_r^{s-1}C_{h_1},\sigma_*)$ with $\tilde C_r = \frac32 C_r$.
\end{enumerate}
\end{theorem}

In the following, for the same reasons bringing to the
choice~\eqref{e.sigma.1}, we will assume $\sigma_*=\sigma_0/4$, as in
the last two sub-points of Theorem~\ref{prop.gen}. Hence, from
\eqref{e.mu.sigma0} and the previous setting of $\sigma_*$ one gets
the relation
\begin{equation}
\label{e.mu.sigma*}
\mu = e^{-\sigma_*}\ ,
\end{equation}
and it is possible to give the following values of some of the
constants involved in the above Theorem:
\begin{equation}
  \label{const.prop.gen}
  \begin{aligned}
    r_* &= \frac{\Omega} {24C_{\zeta_0}}f(\mu)\ ,\qquad
    f(\mu):=\frac{(1-\mu^4)(1-\mu^3)}{\mu^2} \\ \gamma &=
    2\Omega\Bigl(1-\frac{r}{2r_*}\Bigr) \ , \\ C_* &= \frac{3
      C_{h_1}}{\gamma(1-\mu^4)(1-\mu^3)} \ .
  \end{aligned}
\end{equation}
By noticing that condition \eqref{e.muperr} implies
\begin{displaymath}
\Omega<\gamma<2\Omega\ ,
\end{displaymath}
we obtain that $C_*$ {\em essentially} depends on $\mu$, through
$\sigma_0$ and ${\frac{C_{h_1}}{\Omega}}$
\begin{displaymath}
\frac{3 C_{h_1}}{2\Omega(1-\mu^4)(1-\mu^3)} < C_* < \frac{3
  C_{h_1}}{\Omega(1-\mu^4)(1-\mu^3)} \ ;
\end{displaymath}
and this provides $C_r=C_r(r,\mu)$ and $R^* = R^*(r,\mu)$ with
\begin{displaymath}
\frac{\partial C_r}{\partial r}>0   \ ,
\qquad
\frac{\partial C_r}{\partial \mu}>0 \ ,
\qquad\qquad
\frac{\partial R^*}{\partial r}<0   \ ,
\qquad
\frac{\partial R^*}{\partial \mu}<0 \ .
\end{displaymath}

In the forthcoming application, developed in Section \ref{s:3},
instead of fixing $\mu$ as the main (small) parameter like in Theorem
\ref{prop.gen}, we decide to pick the order $r\geq 1$ of the normal
form, and thus the length of the time scale, as the principal
parameter.

As a consequence, by inverting the function $f(\mu)$ in the first of
\eqref{const.prop.gen}, the normal form \eqref{e.Ham.r} holds for all
$\mu<\mu^*(r)$, with
\begin{equation}
\label{e.mu.star}
\mu^*(r):= f^{-1}\tond{\frac{24 C_{\zeta_0}r}{\Omega}}\ .
\end{equation}
Thus, for any $\mu<\mu^*$ we have $R^*(r,\mu) > R^*(r,\mu^*)$. We then
take a threshold $R_*(r)$ for the norm which is uniform with $\mu<\mu^*$
\begin{equation}
\label{e.R.star}
R_{*}(r) := R^*(r,\mu^*(r))\ .
\end{equation}
We summarize the new conditions on the parameters as follows
\begin{align}
\label{e.all.sm}
r &\geq 1\ ,\nonumber\\
\mu &< \mu_1^*(r)\ ,\qquad\Leftrightarrow\qquad c < c_1^*(r)\ ,\\
R &< R_{*}(r)\ .\nonumber
\end{align}

The normal form Theorem \ref{prop.gen} immediately gives the almost
invariance of $H_\Omega$ and $\Z$, which we here formulate
{(see \cite{PalP14}, proof of Corollary~1)} in the
transformed variables $\tilde z=\T_\Chi(z)$

\begin{corollary}
\label{c.Hom.K.var}
Let us take $\tilde z(0)\in B_{\frac49R}$ and let $\tau>0$ be the
escape time of the orbit $\tilde z(t)$ from $B_{\frac23 R}$. Then, for
all times $|t|<\tau$, the approximate integrals of motion $H_\Omega$
and $\Z$ fulfill
\begin{align*}
  |H_\Omega(\tilde z(t))-H_\Omega(\tilde z(0))| &\leq
  \Cst \qquad\;\; \frac{C_{h_1} \Omega}{(1 - \mu)^2} \qquad\;\;
  \;R^4 \tond{\frac23 R^2 C_r}^r \; |t|\ ,
\\
  |\Z(\tilde z(t))-\Z(\tilde z(0))| &\leq 
  \Cst \; \frac{C_{h_1} (\mu C_{\zeta_0} + C_{h_1} R^2)}{(1 - \mu)^2}
  \;R^4 \tond{\frac23 R^2 C_r}^r \; |t|\ .
\end{align*}
\end{corollary}

%
%
%
%

\section{Stability of true and approximated FPU-KG breathers}
\label{s:3}

Let us know denote the normal form part of $H^{(r)}$ --
see~\eqref{e.Ham.r} -- as
\begin{equation}
\label{e.K}
K:=H_\Omega + \Z\ ,\qquad\qquad  \Poi{H_\Omega}{\Z}=0\ ,
\end{equation}
so that the transformed Hamiltonian $H^{(r)}$ can be split as $H^{(r)}
= K + P^{(r+1)}$, and the corresponding Hamilton equations are
\begin{equation}
\label{e.Ham.eq}
\dot z = X_K(z) + X_{P^{(r+1)}}(z)\ .
\end{equation}

The Hamiltonian $K$ (the normal form) looks naturally as the
Hamiltonian of a Generalized discrete Non Linear Schr\"odinger
equation (GdNLS), with $H_\Omega$ in the role of the additional
conserved quantity; an explicit expression of the leading terms of $K$
is available in the Introduction.

As a first, and intermediate, application of such a normal form, we
give an approximation result for the original system~\eqref{e.H}: we
show that for sufficiently small couplings its dynamics stays close
for long times to a closed trajectory in the phase space, provided its
initial datum is also close enough to such an object. This trajectory
is not an orbit of the original system, but it is a breather of the
GdNLS model. The theorem we formulate actually contains, as a first
point, and then exploits, an existence and stability result for the
GdNLS breather itself with respect to the GdNLS dynamics. Such a first
part, despite the generalized nature of the model, is not unexpected
in the anti-continuous limit. The other point, actually the long time
control for the (FPU-)KG model, is less trivial and indeed it is
strongly based on our normal form result.

As a second application, we obtain a result of stability of the true
breather for the original system~\eqref{e.H}, based on the observation
that in the anti-continuous limit there always exists a true breather
which is close enough, with respect to the greatest parameter $c$, to
the approximated one. Thus, the stability we get is actually due to
the stability of the GdNLS orbit.

\subsection{Stability of approximated FPU-KG breathers}
\label{ss:gdnls.br}

Since we base the existence part on the anti-continuous limit, let us
denote by $\tilde\psi_0$ the $0^{\rm th}$-site excitation in the
transformed coordinates $({\tilde x}_j,{\tilde y}_j)$, i.e.  {
\begin{equation}
\label{e.def.psi0}
\tilde\psi_0 := \{({\tilde x}_j,{\tilde y}_j)_{j=0,\ldots,N-1} \ \colon\ \ \ {\tilde x}_0=\rho
,\ \ {\tilde y}_0=0,\ \ {\tilde x}_j={\tilde y}_j=0 \ \forall j\neq0\} \ ,
\end{equation}
}
which is indeed the profile of an initial datum belonging to a
periodic orbit {$\Orb(\tilde\psi_0)$} (trivially a
breather) for the uncoupled system with $a=b=\mu=0$
(see~\eqref{e.constr} below), and for every fixed value of $\rho$. A
consistent choice for the values of $\rho$ will be made later.

\begin{theorem}
\label{p.orb.contr}

Given $r$ and $R$ fulfilling \eqref{e.all.sm},
there exists $c_*(r,R)$ such that, for any $c<c_*$:
\begin{enumerate}

\item there exist a profile $\tilde\psi_{a,b}$ and a frequency
  $\lambda_{a,b}$ such that $\tilde\psi_{a,b} e^{\Im \lambda_{a,b} t}$
  is a Breather solution for the GdNLS \eqref{e.K} with
  $\norm{\tilde\psi_{a,b}}=R/6$ and
  \begin{equation}
    \label{e.ex.psi.ab}
    \norm{\tilde\psi_{a,b}-\tilde\psi_0} \leq \Cst\,\mu\ .
  \end{equation}

\item let us define
\begin{equation}
\label{e.inv.psi.ab}
\psi_{a,b}:=\T_{\Chi}^{-1}{\tilde\psi_{a,b}}\ .
\end{equation}
For any $0<\eps\ll R^2$ there exists $\delta(\eps)$ such that the
(piece of) orbit ${\Orb}(\phi):=\{\phi(t) \ :\ |t| \leq T_{\eps,r,R}
\,,\ \phi(0)=\phi \}$, solution of~\eqref{e.Ham.eq}, satisfies
\begin{equation*}
  \norm{\phi-\psi_{a,b}}<\delta
  \quad\Longrightarrow\quad
  d_H\tond{{\Orb}(\phi),{\Orb}(\psi_{a,b})} < \eps
  \ .
\end{equation*}
where
\begin{equation}
\label{e.times}
T_{\eps,r,R} := C_T \frac{\eps^2}{{R^4}}
{\tond{C_{**}Rr}^{-2r}}\ ,
\end{equation}
with $C_T$ a suitable constant independent on $\eps,\,r$ and $R$,
{and $C_{**}:=8\sqrt{\frac{2C_*}{3}}$.}
\end{enumerate}
\end{theorem}

We could rephrase the result as follows: for small but non vanishing
coupling $\mu$, if we start close enough to the trajectory ${\Orb}(
\psi_{a,b})$ of a GdNLS Breather, we stay close to it (actually in a
small tubular neighborhood of it) for long times.

The proof of the Theorem is made of three steps, which are discussed
in the following subsections: first the existence of a breather for
the GdNLS with a continuation from the $\mu=0$ limit, then its orbital
stability exploiting the exact conservation of $H_\Omega$ (or
equivalently of $\Z$) for the Hamiltonian $K$, and as a last step the
control of the time scale needed to see the effect of the remainder
$P^{(r+1)}$ once the dynamics taken into account is that of the
original system.

\subsubsection{Existence of Breather solutions for the GdNLS}

We denote with ${\Sph}\subset\Ph$ the sphere ${\Sph} :=
\{z\in\Ph\,\big|\,H_\Omega(z)=\rho^2\}$ of (small) radius
$\rho<R_*$. The proof of the first part of Theorem~\ref{p.orb.contr}
is given by the Proposition below, setting $\rho=R/6$.

\begin{proposition}
\label{p.exist.b}
Given $\rho<R_*$, there exists a threshold $c_2^*(\rho)$ and a function
$G:(a,b) \mapsto \tilde\psi_{a,b}$, which belongs to
$\C^1([0,c_2^*)\times [0,c_2^*),\Sph)$, such that $G(0,0)=\tilde\psi_0$
    and
\begin{equation}
\label{e.sol.b}
d\Z\big|_{{\Sph}} (\tilde\psi_{a,b},a,b) = 0\ .
\end{equation}
Moreover, $\tilde\psi_{a,b}$ is close to $\tilde\psi_0$
\begin{equation}
\label{e.close.b}
\norm{\tilde\psi_{a,b}-\tilde\psi_0} \leq \Cst\, \mu\ .
\end{equation}
\end{proposition}

A formal proof of Proposition~\ref{p.exist.b} is deferred to the
Appendix. As we said above, the idea for existence and localization is
to exploit a continuation from the uncoupled limit.  If $\mu=0$, the
model $K$ reduces to a system of $N$ uncoupled an-harmonic oscillators
which admits {$\tilde\psi_0$} (see~\eqref{e.def.psi0}
above) as a local extremizer of the constrained problem
\begin{equation}
\label{e.dec.constr}
\lambda X_{H_\Omega}(z) = X_{\Z}(z)\ ,\qquad \Z:=Z_1+\ldots+Z_r\ ,
\end{equation}
where
\begin{displaymath}
  \zeta_s(\tilde x,\tilde y) :=
      c_s({\tilde x}_0^2+{\tilde y}_0^2)^{s+1}\ .
\end{displaymath}
Indeed, for $\mu=0$ (which means $a=b=0$), the first linear
transformation becomes the identity and all the resonant normal form
construction reduces to $N$ identical Birkhoff normal forms for a
single an-harmonic oscillator. This means that
$\sigma_0=\sigma_*=\infty$ and $\gamma=\Omega=1$ and $c_s$ fulfill
\begin{equation}
\label{e.cs}
\begin{cases}
|c_1| = C_{h_1}\\
|c_s| \leq \frac{C_{h_1}}{s}C_r^{s-1}\ ,\qquad s=2,\ldots,r
\end{cases}\ .
\end{equation}
Moreover, {from its definition in
  Theorem~\ref{prop.gen},} $C_r = {\cal O}\tond{r^2C_{h_1}}$. The
uncoupled constrained problem \eqref{e.dec.constr} reads explicitly
\begin{equation}
\label{e.constr}
\begin{cases}
2\lambda_0 {\tilde x}_j = 4 {\tilde x}_j \sum_{s=1}^r s c_s ({\tilde x}_j^2+{\tilde y}_j^2)^{s} 
\\
2\lambda_0 {\tilde y}_j = 4 {\tilde y}_j \sum_{s=1}^r s c_s ({\tilde x}_j^2+{\tilde y}_j^2)^{s}
\\
\sum_j {{\tilde x}_j^2+{\tilde y}_j^2}=\rho^2
\end{cases}\ ,
\qquad\Rightarrow\qquad
\tilde\psi_0(t) = \tilde\psi_0 e^{\Im \lambda_0 t} \ ,
\end{equation}
with the Lagrange multiplier $\lambda_0 = 2 \sum_{s=1}^r s c_s
\rho^{2s}$ satisfying
\begin{equation}
\label{e.lambda0.est}
|\lambda_0|\geq 2 C_{h_1}\rho^2 - 2 \sum_{s=2}^r s |c_s| \rho^{2s} >
\frac57 C_{h_1} \rho^2\ ,
\end{equation}
where \eqref{e.cs} and the condition \eqref{e.R.sm1} have been used.
The constrained Hessian $M$ is a block-diagonal matrix with blocks
$M_j$; when evaluated on $\tilde\psi_0$, for any $j\not=0$ the block
is $M_j(\tilde\psi_0)=-2\lambda_0\Id$, while for $j=0$ its spectrum is
$\sigma(M_0(\tilde\psi_0))=\{0,{\cal O}(C_{h_1} \rho^2)\} =
\{0,4\lambda_0\}$ with associated eigenvectors $(-{\tilde y}_0,{\tilde x}_0)$ for the
Kernel and $({\tilde x}_0,{\tilde y}_0)$ for the positive direction. This easily proves
that $M$ is definite in all the directions transverse to the orbit
generated by the Hamiltonian field
$X_{H_\Omega}=\Omega(-{\tilde y}_0,{\tilde x}_0,0,\ldots,0)$.
Then\footnote{As clearly explained for example in \cite{Wei99},
  Section 5.} for $\mu/\rho^2$ small enough the Implicit Function
Theorem (IFT) can be applied to uniquely continue $\tilde\psi_0$ to a
solution $\tilde\psi_{a,b}$ of $\Z$ constrained to the level surfaces
of $H_\Omega$ which generates a breather evolution
${\Orb}\tond{\tilde\psi_{a,b}}$
\begin{equation}
\label{e.Br.gdnls}
\begin{cases}
  \lambda_{a,b}\nabla H_\Omega = \nabla \Z
  \\
  \sum_j {{\tilde x}_j^2+{\tilde y}_j^2}=\rho^2
\end{cases}\ ,
\qquad\Rightarrow\qquad
{\Orb}\tond{\tilde\psi_{a,b}} = \tilde\psi_{a,b} e^{\Im \lambda_{a,b} t}\ .
\end{equation}

Concerning the exponential localization, we first remark that it is a
meaningful property also in finite dimension since we aim at estimates
uniform in $N$. From a technical point of view, in the infinite
dimensional case ($N=\infty$) the exponential decay of the amplitudes
${\tilde x}_j^2+{\tilde y}_j^2$ of \eqref{e.Br.gdnls} can be obtained
for example with the IFT on the (Hilbert) phase space
$\ell^2_\sigma\times\ell^2_\sigma$ of square summable sequences with
an exponential decay. Alternatively, it can be obtained by homoclinic
orbits, as in \cite{QinX07}, or by some properties of the inverse of a
Tridiagonal linear operator (still in the IFT framework), as in
\cite{MacA94}.

In the finite case ($N<\infty$), we can speak of an asymptotic
exponential decay as $N$ grows arbitrarily large. For the sake of
simplicity, we prefer showing the proof in the energy norm.  However,
we stress here that the proof holds also when using a norm with
exponential weights, i.e. the finite dimensional subspace of
$\ell^2_\sigma$. In this case, the coercivity constant of the
constrained extremizer is proportional to $e^{-2\sigma}$, hence the
threshold $\mu^*(\rho,\sigma)$ will be a decreasing function of the
decay rate $\sigma$.

%
%

\subsubsection{Orbital stability of the GdNLS breather}

\begin{proposition}
\label{p.orb.stab}
Let $\tilde\psi_{a,b}$ be given by Proposition~\ref{p.exist.b}. For
any positive $\eps\ll \rho^2$ there exists a positive $\delta(\eps)$
such that the orbit ${\Orb}(\tilde\phi):=\{\tilde\phi(t) \ :\ t\in\RR
\,,\ \tilde\phi(0)=\tilde\phi \}$ of the GdNLS (see~\eqref{e.K}),
satisfies
\begin{equation*}
  \norm{\tilde\phi-\tilde\psi_{a,b}}<\delta
  \quad\Longrightarrow\quad
  d_H\tond{{\Orb}(\tilde\phi),{\Orb}(\tilde\psi_{a,b})} < \eps
  \ .
\end{equation*}
\end{proposition}

Given the above use of a constrained critical point formulation to get
the existence of the Breather, the Lyapunov stability in the energy
norm of such an orbit follows once we verify that the Breather is
still an extremizer of $\Z$ in all the constrained transverse
directions and we exploit the fact that $H_\Omega$ and $\Z$ are exact
constants of motion for $K$. The detailed proof is deferred to the
Appendix.

%
%

\subsubsection{Orbital stability of the approximated FPU-KG breather}

Here\footnote{From this subsection on, we will often use $d(A,B)$, as
  defined in~\eqref{e.dist.def}, in the particular case of the set $A$
  given by a single point.} we want to use the existence and stability
of $\tilde\psi_{a,b}$ of Propositions~\ref{p.exist.b}
and~\ref{p.orb.stab} in order to prove the second part of
Theorem~\ref{p.orb.contr}. For this reason, we start taking
\begin{equation}
c_*(r,R) := \min\{c_1^*,c_2^*\}\ ,
\end{equation}
where $c_{1,2}^*$ are introduced in \eqref{e.all.sm} and in
Proposition~\ref{p.exist.b}.

We initially remark that the result is first formulated for the
transformed Hamiltonian~\eqref{e.K}. In order to give the
corresponding statement in the original variables, one has to recall
that the canonical transformation $\T_{\Chi}$ is a perturbation of the
identity \eqref{e.def.Tchi}, hence
\begin{displaymath}
\T_{\Chi}(z) = z + w(z)\qquad\qquad \norm{w(z)} = {\cal O}(\norm{z}^3)\ ,
\end{displaymath}
which implies that $\T_{\Chi}$ (and its inverse) is locally Lipschitz
in any ball $B_R$ sufficiently close to the origin, with a constant $L
= {\cal O}(1)$. Thus the control in the transformed variables can be
transferred to the original ones using the Lipschitz constant of
$T_\Chi^{-1}$
\begin{displaymath}
d(\phi(t),{\Orb}(\psi_{a,b})) = \inf_{\tilde
  w\in{\Orb}(\tilde\psi_{a,b})}\norm{\T_{\Chi}^{-1} \tilde \phi(t) -
  \T_{\Chi}^{-1} \tilde w} \leq L
d(\tilde\phi(t),\Orb(\tilde\psi_{a,b}))\ .
\end{displaymath}

Let us work, then, in the transformed variables $\tilde z$. Our
original system is in the normal form \eqref{e.Ham.eq}, thus
$H_\Omega$ and $\Z$ are only approximate integrals of motion. Hence
the drift from the tubular neighborhood of ${\Orb}(\tilde\psi_{a,b})$
is bounded by the variation of $H_\Omega$ and $\Z$. We assume
$\phi(0)\in B_{R/3}$ and $\norm{\tilde\psi_{a,b}}=R/6$ (hence
$\tilde\psi_{a,b}\in B_{\frac23 R}$) where $R$ satisfies
\eqref{e.all.sm}. We define
\begin{displaymath}
\tilde\phi(t) := \T_\Chi\phi(t)\ ,\qquad\qquad \tilde\phi(0) :=
\T_\Chi\phi(0)\ .
\end{displaymath}
The variations of $H_\Omega$ and $\Z$ in the transformed variables are
controlled by Corollary \ref{c.Hom.K.var}, at least as long as
$\tilde\phi(t)\in B_{\frac23R}$. 

The idea is to work as in the proof of Proposition~\eqref{p.orb.stab}
in the Appendix. Let us assume that $\tilde\phi(t)\in\U$ (which is
true for $t<T_{\U}$ for some $T_{\U}$ if $\tilde\phi(0)$ is
sufficiently close to $\tilde\psi_{a,b}$), where $\U$ is the tubular
neighborhood (defined in the above mentioned proof). In such a
neighborhood the orbital distance from $\tilde\psi_{a,b}$ can be
related to the variations of $H_\Omega$ and $\Z$ as in
\eqref{e.dist.1}. Then it follows immediately
\begin{align*}
d(\tilde\phi(t),\Orb(\tilde\psi_{a,b})) &<
\sqrt{c_3|H_\Omega(\tilde\phi(t))-H_\Omega(\tilde\phi(0))| +
  \frac{c_4}{C_\mu}|\Z(\tilde\phi(t)) - \Z(\tilde\phi(0))|} +\\
&+\sqrt{c_3|H_\Omega(\tilde\phi(0))-H_\Omega(\tilde\psi_{a,b})| +
  \frac{c_4}{C_\mu}|\Z(\tilde\phi(0)) - \Z(\tilde\psi_{a,b})|}=:A+B\ .
\end{align*}
with $C_\mu$ defined at the beginning of the proof of
Proposition~\ref{p.orb.stab}. By using both $(C_{\zeta_0}\mu+C_{h_1}
R^2)/C_\mu={\cal O}(1)$ and Corollary \ref{c.Hom.K.var}, we obtain
\begin{equation}
\label{e.dist.2}
A^2 \leq \Cst \frac{C_{h_1} \Omega}{(1-\mu)^2}
{\;R^4 \tond{\frac23 R^2 C_r}^r \; |t|\ ,}
\end{equation}
which gives
\begin{displaymath}
A<\frac{\eps}{2L}\ ,\qquad\qquad |t|\leq T_{\eps,r,R}\ ,
\end{displaymath}
with a suitable choice of $C_T$. On the other hand, the distance in
the original coordinates can be bounded by exploiting the (local)
Lipschitz constant $L$
\begin{displaymath}
d(\tilde\phi(0),\tilde\psi_{a,b}) \leq L d(\phi(0),\psi_{a,b} )\leq
L\delta(\eps)\ ,
\end{displaymath}
thus for any $\eps$ there exists $\delta(\eps)$ such that
\begin{displaymath}
d(\phi(0),\psi_{a,b} )<\delta \qquad\Rightarrow\qquad
B<\frac{\eps}{2L}\ .
\end{displaymath}
We can collect all the previous estimates to get
\begin{displaymath}
d(\tilde\phi(t),\Orb(\tilde\psi_{a,b})) < \frac{\eps}{L}\ ,
\end{displaymath}
which ensures $\tilde\phi(t)\in\U$ and yields to
\begin{equation}
\label{e.dist.3}
d(\phi(t),{\Orb}(\psi_{a,b})) < \eps \ll R^2\ ,
\end{equation}
for all $|t|\leq T_{\eps,R,r}$, i.e. we have
$d(\Orb(\phi),{\Orb}(\psi_{a,b})) < \eps$.  The same arguments of the
final part of the proof of Proposition~\ref{p.orb.stab} apply here, so
we can get also $d({\Orb}(\psi_{a,b}),\Orb(\phi)) < \eps$ and conclude
the estimate.

As a final comment, we observe that for the time scale considered, the
orbits $\Orb(\phi)$ and $\Orb(\tilde\phi)$ remain in the domains of
definition of the normal form. Indeed since
\begin{displaymath}
\norm{\T_{\Chi}^{-1}{\Orb}(\tilde\psi_{a,b})} = \norm{\tilde\psi_{a,b}}
+ {\cal O}(R^3)=\frac{R}6+ {\cal O}(R^3)\ ,
\end{displaymath}
we obtain
\begin{displaymath}
\norm{\phi(t)} < \frac{R}3\ ,\qquad\qquad \norm{\tilde\phi(t)}<\frac23R\ . 
\end{displaymath}

\qed

%
%

\subsection{Orbital stability of the true FPU-KG Breather: proof of
  Theorem \ref{t.b.kg}}

We recall again that when $c=0$ the normal form transformation
$\T_\Chi$ corresponds to the common Birkhoff change of coordinates
replicated for all the identical an-harmonic oscillators. We denote by
\begin{equation}
\label{e.def.Psi0}
\Psi_0 := {\T}_{\Chi}^{-1} \tilde\psi_0 \ ,
\end{equation}
the one-site excitation in the Birkhoff coordinates, with
$\tilde\psi_0$ given by \eqref{e.def.psi0}. Since
\begin{displaymath}
\norm{\Psi_0 - \tilde\psi_0} \approx \norm{\tilde\psi_0}^3\ ,
\end{displaymath}
the new amplitude will be a $R^3$ deformation of the original one
\begin{displaymath}
\norm{\Psi_0} = \frac{R}6 + {\cal O}(R^3)\ .
\end{displaymath}

At fixed small amplitude, when the coupling parameters $a,\,b$ are
switched on, the one-site periodic orbit $\Psi_0$ can be continued to
(form) a family $\Psi_{a,b}$, provided $c<c_3^*(R)$, as originally
proved in \cite{MacA94}. On the other hand, the reference solution
$\tilde\psi_0$ can be continued to (form) a family $\tilde\psi_{a,b}$
of orbitally stable Breather solutions for the normal form
\eqref{e.K}, as claimed in Proposition~\ref{p.orb.stab}. We recall
that in \eqref{e.inv.psi.ab} we have denoted by $\psi_{a,b}$ the
inverse image of such a family of {\em approximated Breather solution}
for the original FPU-KG model, in the sense of
Proposition~\ref{p.orb.contr}. Due to our initial choice for $\Psi_0$,
the two families initially coincide. Hence, there exists
$c_4^*(r,R)<c_3$, such that for $c<c_4^*$ the two families are
$\mu$-close
\begin{equation}
\label{e.a-close}
\norm{\Psi_{a,b} -{\psi_{a,b}}}<\Cst\,\mu\ .
\end{equation}

With this kind of control we are actually able to close the proof by
the use of the triangle inequality (and this is ultimately the reason
to use the Hausdorff distance). Indeed we control the distance between
$\Orb(\Psi_{a,b})$ and $\Orb(\phi)$ triangulating via
$\Orb(\psi_{a,b})$.

We proceed as follow. We exploit the stability of the GdNLS breathers
to control both $\Orb(\phi)$ and $\Orb(\Psi_{a,b})$. We thus apply
twice the second part of Theorem~\ref{p.orb.contr}, first to
$\Orb(\phi)$:
\begin{equation}
\label{e.tri.1}
\exists \delta(\eps/2)\ :\ \norm{\phi-\psi_{a,b}}<\delta
\qquad\Longrightarrow\qquad
d_H\tond{\Orb(\psi_{a,b}),\Orb(\phi)}<\frac{\eps}2 \ ,
\end{equation}
and then to  $\Orb(\Psi_{a,b})$:
\begin{equation}
\label{e.tri.2}
\exists \delta(\eps/2)\ :\ \norm{\Psi_{a,b}-\psi_{a,b}}<\delta
\qquad\Longrightarrow\qquad
d_H\tond{\Orb(\psi_{a,b}),\Orb(\Psi_{a,b})}<\frac{\eps}2 \ .
\end{equation}
Concerning this second estimate we remark that the period of the
Breather $\Orb(\Psi_{a,b})$ is of order 1 and surely
shorter\footnote{Unless one decides to take { $\eps\ll R^2(Rr)^r$},
  which is not necessary to get a meaningful result of orbital
  stability. In that case, one would get the stability only of a piece
  of the periodic orbit.} that then stability time $T_{\eps,r,R}$; the
Breather itself is thus entirely contained in the tubular
neighborhood.

In order to use implications~\eqref{e.tri.1} and~\eqref{e.tri.2} one
has to ensure the control on the distance of the initial datum from
$\psi_{a,b}$: for $\Psi_{a,b}$ we use~\eqref{e.a-close}, and for
$\phi$ we triangulate again, this time around $\Psi_{a,b}$. More
precisely
\begin{equation*}
  \norm{\phi-\psi_{a,b}} \leq \norm{\phi-\Psi_{a,b}} +
  \norm{\Psi_{a,b}-\psi_{a,b}} < \delta(\eps/2) \ ,  
\end{equation*}
where the first addendum is the one whose smallness we are free to
impose in the statement of the Theorem, and the second can be made as
small as we wish again using~\eqref{e.a-close}.

Thus, provided $c$ is small enough to effectively
use~\eqref{e.a-close}, and we are close enough to $\Psi_{a,b}$ with
our initial datum $\phi$, estimates~\eqref{e.tri.1}
and~\eqref{e.tri.2} hold, so that
\begin{equation*}
  d_H\tond{\Orb(\phi),\Orb(\Psi_{a,b})} \leq
  d_H\tond{\Orb(\phi),\Orb(\psi_{a,b})} +
  d_H\tond{\Orb(\psi_{a,b}),\Orb(\Psi_{a,b})} < \eps.
\end{equation*}
 
\qed

%
%
%
%
%

\section{Appendix}
\label{s:6}

\subsection{Proof of Proposition \ref{p.exist.b}}

Differently from the mostly used technique of Lagrange multipliers
(see e.g.~\cite{Wei99}), we here prefer working locally on the
constraint, thus we make use of a local parametrization of the
manifold with its tangent space. In this way, we still have a
functional defined over a linear euclidean space. After this
preliminary operation, the problem is treated with the usual IFT (see
\cite{AmbP95} and \cite{KolF89}). The geometric part of the proof is
trivial since the phase space is of finite dimension $2N$. However,
the estimates are uniform with $N$.

We consider the tangent space in a point ${\tilde z}\in{\Sph}$, as
defined by $T_{\tilde z}{\Sph} := \{Y\in\Ph\,\big|\,\sum_j {\tilde
  z}_j Y_j=0\} = \inter{{\tilde z}}^\perp$, where $\inter{{\tilde z}}$
represents the linear space generated by ${\tilde z}$. Since
$X_{H_\Omega}({\tilde z})\in T_{\tilde z}{\Sph}$, the set $V:=\{Y\in
T_{\tilde z}{\Sph}\, \big| \,\sum_j Y_j (X_{H_\Omega}({\tilde
  z}))_j=0\}\subset T_{\tilde z}{\Sph}$ is a linear subspace of $\Ph$
(of dimension $2N-2$).

Take ${\tilde z}=\tilde\psi_0$ as in \eqref{e.def.psi0}. The phase
space $\Ph$ can be decomposed into the direct sum of the tangent space
$T_{\tilde\psi_0}{\Sph}$ and its orthogonal direction $\tilde\psi_0$,
and also the tangent space itself can be decomposed into the field
direction $X_{H_\Omega}(\tilde\psi_0)$ and its orthogonal complement
$V$
\begin{displaymath}
\Ph =
T_{\tilde\psi_0}{\Sph}\oplus{\tilde\psi_0}\ ,
\qquad
T_{\tilde\psi_0}{\Sph} = V\oplus X_{H_\Omega}(\tilde\psi_0)\ .
\end{displaymath}
This gives the characterization
\begin{equation}
  V=\graff{(\tilde x,\tilde y)\in\Ph\,\big|\,{\tilde x}_0={\tilde y}_0=0}\ .
\end{equation}

Let us work locally on a neighborhood of $\tilde\psi_0\in{\cal
  S}$. There exist ${\U}(\tilde\psi_0)\subset{\Sph}$ and a function
$f:{\cal W}\subset
T_{\tilde\psi_0}{\Sph}\rightarrow\inter{\tilde\psi_0}$ such that, for
any ${\tilde z}\in{\U}$ there exists $h\in T_{\tilde\psi_0}{\Sph}$
satisfying
\begin{displaymath}
{\tilde z} = P(h):= \tilde\psi_0+h+f(h)\ ;
\end{displaymath}
in rough words, locally the sphere is the graph of a function $f$
defined on the tangent space. The above map is a $\C^2(\W,\U)$
diffeomorphism.  From the previous decomposition of
$T_{\tilde\psi_0}{\Sph}$, it is locally well defined the submanifold
\begin{equation}
\label{e.subman.M}
\M := \{{\tilde z}\in{\U}\, \big| \, {\tilde z}=P(h),\, h\in V\cap\W\}\ .
\end{equation}
By construction we have $T_{\tilde\psi_0}\M = V$.

Since $H_\Omega$ is a preserved quantity for $\Z$, the flow of
$X_{H_\Omega}$ is a symmetry and then
\begin{equation}
d\Z\big|_{{\Sph}} ({\tilde z},a,b) =
0\qquad\Leftrightarrow\qquad d\Z\big|_{\M}(\tilde z,a,b)=0\ .
\end{equation}

We are interested in the problem
\begin{displaymath}
d \Z\big|_{\M} ({\tilde z},a,b) = 0 \ ,
\end{displaymath}
which has the solution $\tilde\psi_0$ for $a=b=\mu=0$. From the local
linear representation of $\M$, we can consider $\Z$ on the linear
space $V$
\begin{displaymath}
\Z(h,a,b) := \Z(P(h),a,b) = \Z\big|_{\M} ({\tilde z},a,b) \qquad h\in
V\cap\W\ ,
\end{displaymath}
which is at least $\C^2(V,\RR)$. We look for a map
\begin{equation}
\label{e.g}
g:(a,b)\in[0,c^*)\times[0,c^*)\mapsto h=g(a,b)\in V\ ,\qquad g(0,0)=0\ ,
\end{equation}
such that $\Z'_h(g(a,b),a,b) =0$; we already know that $\Z'_h(0,0,0) =
0$.

We set the operator $F$
\begin{equation}
\label{e.F}
F(h,a,b) := \Z_h'(h,a,b) = \Z_\psi'(P(h),a,b)P'(h)\ ,
\end{equation}
which, due to Lemma \ref{l.cont.pol}, is $\C^1$ from $V\times\RR^2$ to
$V^*:=L(V,\RR)$ and satisfies
$F(0,0,0)=\Z_\psi'(\tilde\psi_0,0,0)\Id=0$.

The differential $F_h'(0,0,0) = \Z_\psi''(\tilde\psi_0,0,0)\Id$, which
maps $V$ to its dual $V^*$, has the inverse bounded by the constant
$1/C_{h_1} \rho^2$; indeed, we already observed {in subsection~3.1.1}
that, {when $a=b=0$}, the whole orbit generated by $\tilde\psi_0$ is a
constrained strong extremizer, hence the constrained Hessian is
coercive in all the directions transverse to
$X_{H_\Omega}(\tilde\psi_0)$, with coercivity constant $2 C_{h_1}
\rho^2$ in the euclidean norm. A direct computation, which is based on
  the explicit computations developed in subsection~3.1.1, shows
indeed that one has
\begin{equation}
\label{e.coerc}
|\Z_\psi''(\tilde\psi_0,0,0)\Id\quadr{Y,Y}| \geq C_2
\sum_{j\not=0}Y_j^2 \geq {C_2}\norm{Y}^2\ ,\qquad C_2 := 2C_{h_1}
\rho^2\ ,
\end{equation}
hence
\begin{equation}
\label{e.inv.F1}
\norm{\quadr{F_h'(0,0,0)}^{-1}}_{L(V^*,V)} \leq 1/C_2\ .
\end{equation}

Then there exist $\mu_2^*(\rho)$, and hence from \eqref{e.mu} a
$c_2^*=c_2^*(\rho)$, and a function
$g\in\C^1([0,c_2^*)\times[0,c_2^*),V)$ as in \eqref{e.g} such that
    $F(g(a,b),a,b)=0$ with
\begin{displaymath}
\norm{g(a,b)}< \mu\ ,\qquad\qquad |\mu|<\mu_2^*\ .
\end{displaymath}
Furthermore, we recall that the IFT is based on the contraction
Theorem on a closed $\eps$-ball $B_\eps\subset V$ for the operator
$A_{a,b}(h):= h - \quadr{F_h'(0,0,0)}^{-1}F(h,a,b):V\mapsto V$. The
requirements of being a contraction and surjective on $B_\eps$,
implies that $\mu_2^*$ is bounded by the coercive constant $C_2$ in
\eqref{e.inv.F1}.

From the property $f(h)=o(\norm{h})$ of the parametrization $P$, it
immediately follows that the solution $\tilde\psi_{a,b} := P(g(a,b))
=: G(a,b)$ is $\mu$ close to $\tilde\psi_0$
\begin{displaymath}
\norm{\tilde\psi_{a,b}-\tilde\psi_0} \leq \Cst\, \mu\ .
\end{displaymath}

\qed


\subsection{Proof of Proposition \ref{p.orb.stab}}

From the continuity of $\Z''$ we deduce that $\tilde\psi_{a,b}$ is
still a strong extremizer in the direction $V$, with a coercive
constant $C_\mu = \mathcal{O}(\rho^2)$.  Hence the orbit generated by
he flow of $X_{H_\Omega}(\psi)$ is orbitally Lyapunov stable with $\Z$
being the Lyapunov function (see \cite{BamN98,Bam96,Wei99,Wei86}).

In few words (inspired also by Section 8 of \cite{BamN98}, in
particular Lemmas~8.5 and 8.6, although we work in the simplified case
of a finite dimensional phase space), given a generic point of the
orbit $\tilde\eta\in\Orb(\tilde\psi_{a,b})$, there exists a
neighborhood $\W_0$ of $\tilde\eta$ where a suitable set of
coordinates can be introduced. This local representation is based on
the decomposition $P_{\tilde\eta} = \nabla H_\Omega(\tilde\eta) \oplus
V_{\tilde\eta}$ of the hyperplane $P_{\tilde\eta}$ orthogonal to
$X_{H_\Omega}(\tilde\eta)$, for any
$\tilde\eta\in\Orb(\tilde\psi_{a,b})$. More precisely, there exists a
(tubular) neighborhood $\W_0$ of $\tilde\eta$ such that, for any point
$\tilde z\in\W_0$, the hyperplane through $\tilde z$ and orthogonal to
$\Orb(\tilde\psi_{a,b})$ is unique. This plane intersects the periodic
orbit $\Orb(\tilde\psi_{a,b})$ in a point $\tilde\xi$, which can be
obtained as the evolution of $\tilde\eta$ at ``time'' $\varphi$ along
the flow of the periodic orbit. Hence, using the previous notation,
such a plane can be decomposed as $P_{\tilde\xi} = \nabla
H_\Omega(\tilde\xi) \oplus V_{\tilde\xi}$. This implies that $\tilde
z$ can be locally represented by the coordinates
\begin{equation}
  \label{e.loc.coord}
  \tilde z\equiv (\varphi,E,v)\in\RR\times\RR\times V_{\tilde\xi}\ ,
\end{equation}
where $E$ represents the displacement in the $\nabla
H_\Omega(\tilde\xi)$ direction and $v$ the displacement in the
$V_{\tilde\xi}$ direction(s).  Using these local coordinates in order
to represent $\tilde z= \tilde\phi(t)\in\W_0$, the orbital distance of
$\tilde\phi(t)$ from $\Orb(\tilde\psi_{a,b})$ is controlled in $\W_0$
by
\begin{equation}
  \label{e.dist.0}
  d(\tilde\phi(t),\Orb(\tilde\psi_{a,b})) \leq \inf_{w\in
    \Orb(\tilde\psi_{a,b})\cap \W_0} \norm{w-\tilde\phi(t)}\leq c_1 |E(t)| +
  c_2\norm{v(t)}\ ,
\end{equation}
with $c_{1,2}$ depending on $\W_0$.  The first term $|E(t)|$
represents the variation of
$|H_\Omega(\tilde\phi(t))-H_\Omega(\tilde\psi_{a,b})| =
|H_\Omega(\tilde\phi(t))-H_\Omega(\tilde\xi)|$: indeed, being $E(t)$
the coordinate associated to the direction $\nabla
H_\Omega(\tilde\xi)$, with $\tilde\xi\in\Orb(\tilde\psi_{a,b})$, it
controls the displacement orthogonal to $\Sph_{\tilde\xi}$. The second
term $\norm{v(t)}$ is instead related to the variation of
$|\Z(\tilde\phi(t)) - \Z(\tilde\psi_{a,b})| = |\Z(\tilde\phi(t)) -
\Z(\tilde\xi)|$, which controls the $V_{\tilde\xi}$ directions
transverse to the orbit, {\em provided $\norm{v(t)}$ is small
  enough}. Here enters the fact that any point
$\tilde\xi\in\Orb(\tilde\psi_{a,b})$ is a local extremizer for $\Z$
constrained to\footnote{One can define $\M_{\tilde\xi}$ as the
  submanifold tangent to $V_{\tilde\xi}$ as in \eqref{e.subman.M}.}
$\M_{\tilde\xi}$. Indeed, if we take a point $\tilde z\in
V_{\tilde\xi}$ close enough to $\tilde\xi$ (such that $\Z$ almost
coincides with its quadratic part), then a Taylor expansion gives
\begin{displaymath}
  \Z(\tilde z) - \Z(\tilde\xi) = \frac12
  \Z''(\tilde\xi)[\tilde z-\tilde\xi,\tilde
    z-\tilde\xi]+h.o.t\ ,
\end{displaymath}
which provides the bound
\begin{equation}
  \label{e.Z.control}
  \norm{\tilde z-\tilde\xi}^2 = \norm{v}^2\leq \frac3{C_\mu}|\Z(\tilde z) -
  \Z(\tilde\xi)|\ ,\qquad \norm{v}\ll C_\mu\sim \rho^2 \ .
\end{equation}
Thus there exists a neighborhood $\W_1\subset\W_0$ of $\tilde\eta$
such that if $\tilde\phi(t)\in\W_1$ then \eqref{e.dist.0} becomes
\begin{equation}
\label{e.dist.1}
d(\tilde\phi(t),\Orb(\tilde\psi_{a,b})) \leq
\sqrt{c_3|H_\Omega(\tilde\phi(t))-H_\Omega(\tilde\psi_{a,b})| +
  \frac{c_4}{C_\mu}|\Z(\tilde\phi(t)) - \Z(\tilde\psi_{a,b})|}\ .
\end{equation}
with $c_{3,4}$ depending on $\W_1$.  Since $\Orb(\tilde\psi_{a,b})$ is
compact (being homeomorphic to $S^1$), we can cover a whole
neighborhood $\U$ of this orbit with a finite collection (independent
of $N$) of local neighborhoods like $\W_1$ and set of coordinates like
\eqref{e.loc.coord}, such that \eqref{e.dist.1} holds true. Since both
$H_\Omega$ and $\Z$ are continuous (analytic, see Lemma
\ref{l.cont.pol}) constants of motion for $K$, the requirement of
staying in $\U$ is translated in a closeness condition for the initial
datum $\tilde\phi(0)$: there exists $\delta(\eps)$ such that
\begin{displaymath}
d(\tilde\phi(0),\Orb(\tilde\psi_{a,b}))<\delta \quad\Rightarrow\quad
\sqrt{c_3|H_\Omega(\tilde\phi(0))-H_\Omega(\tilde\psi_{a,b})| +
  \frac{c_4}{C_\mu}|\Z(\tilde\phi(0)) - \Z(\tilde\psi_{a,b})|}<\eps\ .
\end{displaymath}

This actually gives $d(\Orb(\tilde\phi),\Orb(\tilde\psi_{a,b}))<\eps$,
i.e. the orbit we aim to control is contained in the tubular
neighborhood of the breather $\tilde\psi_{a,b}$ for the normal form
$K$. To conclude the proof we also need the symmetric control, to
avoid that our orbit, despite being in $\U$, does not actually follow
the whole trajectory of the breather. Indeed, in full generality it
could happen that the orbit goes back and forth only in a section of
$\U$; or it could happen that such a neighborhood is not homotopic to
an $S^1$, e.g. it has an ``eight'' shape, and in that case the orbit
could use the ``connection'' as a shortcut to follow only a part of
the orbit without leaving $\U$. In our case these problems do not
arise: indeed the GdNLS breather is given by the action of
$e^{i\lambda t}$ on $\Sph$, i.e. it is a maximal circle on a sphere
whose radius is of order $\rho$. On the other hand, $\U$ is the
cartesian product of the breather and a disc of co-dimension one,
whose radius has to be of order $\eps$ which is constrained to be
smaller than $\rho^2$. As a first consequence $\U$ is necessarily
homotopic to an $S^1$. Moreover the component of the vector field
transverse to the disc is a small perturbation of the vector field in
the point of the breather orbit which lies in the disc itself. It is
thus not possible for any orbit in $\U$ to stop flowing along the
tubular neighborhood, and this happen in a time which is a small
perturbation of the period of the breather.

\noindent The above arguments allow us to get also
$d(\Orb(\tilde\psi_{a,b}),\Orb(\tilde\phi))<\eps$, and this concludes
the proof.

\qed

\paragraph{Acknowledgments: }
We warmly thank Dario Bambusi for suggesting to separate this
application from the normal form result~\cite{PalP14}, and for other
useful comments. This research is partially supported by MIUR-PRIN
program under project 2010 JJ4KPA (``Te\-orie geome\-triche e
anali\-tiche dei sistemi Hamilto\-niani in dimensioni finite e
infi\-nite'').


\def\cprime{$'$} \def\i{\ii}\def\cprime{$'$} \def\cprime{$'$}

\end{document}